\documentclass{amsart}

\usepackage{amsmath,amssymb,amsthm}

\newtheorem{prop}{Proposition}[section]
\newtheorem{thm}[prop]{Theorem}
\newtheorem{lem}[prop]{Lemma}

\theoremstyle{definition}

\newtheorem{rem}[prop]{Remark}

\newtheorem*{ack}{Acknowledgement}

\def\co{\colon\thinspace}

\newcommand{\C}{\mathbb C}

\newcommand{\rmd}{\mathrm d}
\newcommand{\D}{\mathbb D}

\newcommand{\rme}{\mathrm e}

\newcommand{\Hp}{\mathbb H}

\newcommand{\rmi}{\mathrm i}

\newcommand{\N}{\mathbb N}

\newcommand{\R}{\mathbb R}

\newcommand{\lra}{\longrightarrow}
\newcommand{\ra}{\rightarrow}

\DeclareMathOperator{\dist}{\mathrm{dist}}

\DeclareMathOperator{\length}{\mathrm{length}}
\DeclareMathOperator{\loc}{\mathrm{loc}}

\begin{document}

\author{Urs Frauenfelder}
\address{Institut f\"ur Mathematik, Universit\"at Augsburg}
\email{Urs.Frauenfelder@math.uni-augsburg.de}
\author{Kai Zehmisch}
\address{Mathematisches Institut, Westf\"alische Wilhelms-Universit\"at M\"unster}
\email{Kai.Zehmisch@uni-muenster.de}

\title[Gromov compactness]
{Gromov compactness for holomorphic
  discs with totally real boundary conditions}

\date{16.04.2015}

\begin{abstract}
  We prove
  that a sequence of holomorphic discs
  with totally real boundary conditions
  has a subsequence
  that Gromov converges
  to a stable holomorphic map
  of genus zero with connected boundary
  provided that the sequence is bounded
  and has bounded energy.
\end{abstract}

\thanks{UF partially supported by the Basic Research fund
  2013004879 of the Korean government.
  ZK is partially supported by DFG grant ZE 992/1-1.}

\maketitle


\section{Introduction\label{sec:intro}}

The study of Lagrangian submanifolds in symplectic topology
uses holomorphic discs in an essential way.
A natural question is the convergence of
sequences of parametrized holomorphic discs
with boundary values in a given Lagrangian submanifold.
By the elliptic nature of the Cauchy-Riemann equation
a uniform bound with respect to the $C^1$-topology
implies convergence of a subsequence including
all partial derivatives.
A $C^0$-bound usually is given by
the geometric problem under consideration
so that the actual question
is about a bound on the gradient.

In general there is no gradient bound.
This is because the group of holomorphic
automorphisms of the unit disc,
which acts on all holomorphic discs by
reparametrization,
is not compact.
On the other hand, the symplectic energy,
which is the integral of the symplectic form
over the holomorphic disc,
is preserved under this action.
This implies uniform $L^2$-bounds
on the gradient provided the holomorphic
discs represent the same relative homotopy class.
In fact, as observed by Sacks-Uhlenbeck
\cite{saul81} and Gromov \cite{grom85},
locally up to reparametrisations
one obtains $C^{\infty}$-convergence
using the so-called bubbling off analysis.
The limit object consists of finitely many
holomorphic discs and spheres
glued along finitely many points
so that the relative homotopy class is preserved.
The so-called Gromov convergence is described in
\cite{fraun08}
and content of the present article.

Holomorphic discs appear in contact geometry
in conjunction with questions about
fillability as in
\cite{bour09,elia90,elho94,elpo96,
  geizeh10,geizeh13,geizeh14a,geizeh14b,gnw14,grom85,
  hind97,maniewen13,nieder06,
  ye98,zehm03}
or the strong Weinstein conjecture
\cite{albhof09,geizeh13,hofe93,niedrech11}.
One considers contact manifolds
that are the boundary of a symplectic manifold.
The boundary condition for the holomorphic discs
is posed with respect to a totally real
submanifold that is not Lagrangian.
Because of the integrability
of the characteristic distribution
in the complement of the singular set
uniform gradient bounds
along the boundary curve of the holomorphic discs
can be obtained by geometric arguments.
In order to unify the argumentation and
for a deeper understanding
of the Reeb dynamics on a contact manifold
Gromov convergence of holomorphic discs
is required for totally real boundary conditions.

A description how a sequence of holomorphic
discs with uniform $W^{1,2}$-bounds
and totally real boundary conditions
degenerates in the sense of Sacks-Uhlenbeck
can be found in \cite{ye94}.
In view of the methods used in symplectic topology
a description in the language of stable maps
due to Kontsevich \cite{kont95}
as in \cite{fraun08} is required.
Approaches that use nodal surfaces similar
to Gromov's original description
in \cite{grom85} can be found in \cite{ursfu14,ivshe02}.
An alternative concept makes use of the so-called bubble trees,
see \cite{fraun08,mcsa04}.
To provide an extension of this
to holomorphic discs with totally real boundary conditions
is the aim of the article.

\subsection{Formulation of the theorem}

Let $(M,\omega)$ be a symplectic manifold
of dimension $2n$.
Let $J$ be an almost complex structure on $M$
that is tamed by $\omega$.
We consider a {\bf totally real}
submanifold $L$,
which by definition is an
$n$-dimensional submanifold $L$ of $M$
such that $TL\cap JTL$ is the zero section.
All closed Lagrangian submanifolds
and the boundaries in the filling problem
are included.
A {\bf holomorphic disc}
\[
u\co(\D,\partial\D)\lra (M,L)
\]
is a smooth map
$u\co\D\ra M$
defined on the closed unit disc
$\D\subset\C$
that is a solution of the
{\bf Cauchy-Riemann equation}
\[
u_x+J(u)u_y=0,
\]
where $z=x+\rmi y$
are conformal coordinates on $\D$
and $u_x$, $u_y$ stand for 
$Tu(\partial_x)$, $Tu(\partial_y)$.
We assume that $u$ maps the boundary $\partial\D$
into the totally real submanifold $L$ of $M$.

A sequence $u^{\nu}$
of holomorphic discs is said to be {\bf bounded}
provided that there exists a compact subset of $M$
that contains $u^{\nu}(\D)$ for all $\nu\in\N$
and a compact subset of $L$
that contains $u^{\nu}(\partial\D)$ for all $\nu\in\N$.
We say that $u^{\nu}$
has {\bf bounded energy} if
\[
\sup_{\nu\in\N}E(u^{\nu})<\infty,
\]
where the {\bf energy}
of a holomorphic disc $u$ is defined to be
\[
E(u)=\int_{\D}u^*\omega.
\]
We remark that by the tameness condition
the symmetrisation of $\omega(\,.\,,J\,.\,)$
is a metric on $M$.
With respect to this metric the energy
of a holomorphic disc $u$
equals the {\bf Dirichlet energy}
\[
E(u)=\frac12\int_{\D}|\nabla u|^2,
\]
where the {\bf energy density} is defined by 
\[
|\nabla u|^2=|u_x|^2+|u_y|^2.
\]

\begin{thm}
  \label{mainthm}
  Any bounded sequence of holomorphic discs
  with bounded energy
  has a subsequence that Gromov converges
  to a stable holomorphic map
  of genus zero with connected boundary.
  The relative homotopy class is preserved
  under Gromov convergence.
\end{thm}

The definitions of Gromov convergence and stable maps,
which we took from \cite{fraun08},
are repeated in Sections \ref{subsec:stablemaps}
and \ref{subsec:gromovconv} below.
The proof of the theorem is given in
Section \ref{sec:bubbcon}.
The primary idea is to Lagrangefy the boundary condition
and is described in Section \ref{sec:lagrangefy}.
In Section \ref{sec:general}
we will indicate several generalizations.

\subsection{Holomorphic spheres\label{subsec:holomspher}}

A holomorphic sphere is a smooth map
$u\co\C P^1\ra M$
that solves the Cauchy-Riemann equation
$J\circ Tu=Tu\circ\rmi$,
where $\rmi$ denotes the complex structure
on $\C P^1=\C\cup\infty$
obtained by adding the point at infinity to $\C$.
The energy of a holomorphic sphere is
\[
E(u)=\int_{\C P^1}u^*\omega=\frac12\int_{\C P^1}|T u|^2,
\]
where the Dirichlet energy integral is taken
with respect to the area form
induced by the Fubini-Study metric.

\subsection{Stable maps\label{subsec:stablemaps}}

A {\bf tree} is a finite $1$-dimensional
simply connected cell complex $T$.
The set of $0$-cells (or vertices) is denoted by $V$
and the set of $1$-cells (or edges) by $E$.
An edge is uniquely described by the set of its vertices.
To each vertex $\alpha$
we attach a holomorphic map
\[
u_{\alpha}\co(\Sigma_{\alpha},\Gamma_{\alpha})
\lra (M,L)
\]
that is required to be either
a holomorphic disc $u_{\alpha}\co(\D,\partial\D)\ra (M,L)$,
a holomorphic sphere $u_{\alpha}\co\C P^1\ra M$, or
a holomorphic sphere $u_{\alpha}\co(\C P^1,\infty)\ra (M,L)$
that maps the point at infinity to $L$.

To each edge $\{\alpha,\beta\}$
we associate two points
$z_{\alpha\beta}\in\Sigma_{\alpha}$ and
$z_{\beta\alpha}\in\Sigma_{\beta}$
which we call {\bf nodal points}.
If $z_{\alpha\beta}$ is a boundary
point on $\Sigma_{\alpha}=\D$
then $\Sigma_{\beta}=\D$ is a disc
and $z_{\beta\alpha}$ is a boundary point too.
In other words,
the nodal points corresponding to an edge
are either interior or boundary points;
no interior point can correspond
to a boundary point and vice versa.
We require
that for different edges $\{\alpha,\beta\}$
and $\{\alpha,\gamma\}$ the points $z_{\alpha\beta}$
and $z_{\alpha\gamma}$ on $\Sigma_{\alpha}$
are different.
If the surface $(\Sigma_{\alpha},\Gamma_{\alpha})$
is the pointed $(\C P^1,\infty)$
we require all nodal points on
$\Sigma_{\alpha}$ to be different from $\infty$.

The {\bf boundary tree} $\partial T$ of $T$
is the subtree, whose vertices $\alpha$
correspond to surfaces $\Sigma_{\alpha}$
with $\Gamma_{\alpha}\neq\emptyset$.
We require that $\partial T$ is non-empty.
The boundary tree
consists of a single vertex
if $\Gamma_{\alpha}=\infty$
for some $\alpha\in V$.

The collection of holomorphic maps $u_{\alpha}$
is required to satisfy
$u_{\alpha}(z_{\alpha\beta})=u_{\beta}(z_{\beta\alpha})$
for all nodal points.
Passing to the quotient
of the disjoint union of the $\Sigma_{\alpha}$
by the set of all nodal points
we obtain a nodal surface
of genus zero with connected
boundary
that eventually is collapsed to one point.
The $u_{\alpha}$ pass
to a holomorphic map $\mathbf u$
on the quotient.

We denote the set of nodal points
on $\Sigma_{\alpha}$ by $Z_{\alpha}$
including $\infty$ in the pointed case.
The map $\mathbf u$ is called {\bf stable}
if for all vertices $\alpha$ for which
the map $u_{\alpha}$ is constant
the number of nodal points is
$\#Z_{\alpha}\geq3$
(regardless $\Sigma_{\alpha}$ is a disc or a sphere)
or $Z_{\alpha}$ consists of two points
that do not both lie on the
boundary $\Gamma_{\alpha}=\partial \D$
in the case that 
$\Sigma_{\alpha}=\D$ is a disc.

\subsection{Energy of a bubble tree\label{subsec:energybubbletree}}

Let $\mathbf u=(u_{\alpha})$ be a stable map
with tree $T=(V,E)$.
Removing (the interior of)
an edge $\{\alpha,\beta\}\in E$
decomposes $T$ into two subtrees.
Denote by $T_{\alpha\beta}$ the component
that has $\beta$ as a vertex.
Observe that any subtree of $T$
is of the form $T_{\alpha\beta}$.
The {\bf energy of the bubble tree}
$T_{\alpha\beta}$ is
\[
E_{\alpha\beta}(\mathbf u)=
\sum_{\gamma\text{\,\,is vertex of\,\,}T_{\alpha\beta}}
E(u_{\gamma}).
\]
Similarly,
the {\bf total energy} is
\[
E(\mathbf u)=
\sum_{\alpha\in V}
E(u_{\alpha}).
\]

\subsection{Gromov convergence\label{subsec:gromovconv}}

We say
that a sequence of holomorphic discs
{\bf Gromov converges}
to a stable map $\mathbf u$
(denoting the underlying tree by $T=(V,E)$)
provided that for each $\alpha\in V$
there exists a sequence of
M\"obius transformations $\varphi_{\alpha}^{\nu}$
in the automorphism group of $\Sigma_{\alpha}$
such that the following holds:
If $\Sigma_{\alpha}$ is the disc $\D$
or the unpointed $\C P^1$ we require:
\begin{itemize}
\item 
  {\bf (Rescaling):}
  For all $\{\alpha,\beta\}\in E$
  the sequence
  $\big(\varphi_{\alpha}^{\nu}\big)^{-1}\circ\varphi_{\beta}^{\nu}$
  converges to the constant map $z_{\alpha\beta}$
  in $C^{\infty}_{\loc}$
  on $\Sigma_{\beta}\setminus\{z_{\beta\alpha}\}$.
\item 
  {\bf (Map):}
  For all $\alpha\in V$
  the sequence
  \[
  u^{\nu}_{\alpha}:=u^{\nu}\circ\varphi_{\alpha}^{\nu}
  \]
  converges to $u_{\alpha}$
  in $C^{\infty}_{\loc}$
  on $\Sigma_{\alpha}\setminus Z_{\alpha}$.
\item 
  {\bf (Energy):}
  For all $\{\alpha,\beta\}\in E$
  \[
  E_{\alpha\beta}(\mathbf u)
  =
  \lim_{\varepsilon\ra\infty}
  \lim_{\nu\ra\infty}
  E\big(u^{\nu}_{\alpha};B_{\varepsilon}(z_{\alpha\beta})\big),
  \]
  where $B_{\varepsilon}(z_{\alpha\beta})$ denotes
  set of points in $\Sigma_{\alpha}$
  with distance less than $\varepsilon$
  from $z_{\alpha\beta}$
  and
  \[
  E\big(u^{\nu}_{\alpha};B_{\varepsilon}(z_{\alpha\beta})\big)=
  \int_{B_{\varepsilon}(z_{\alpha\beta})}(u^{\nu}_{\alpha})^*\omega.
  \]
\end{itemize}

If the {\bf degenerate case}
$(\Sigma_{\alpha}=\C\cup\infty,\Gamma_{\alpha}=\infty)$
appears we require in addition
that:

\begin{itemize}
\item 
  {\bf (Rescaling):}
  There exists for each compact
  subset $K\subset\C$
  a natural number $\nu_0$
  such that
  $\varphi_{\alpha}^{\nu}(K)\subset B_1(0)$
  for all $\nu\geq\nu_0$.
\item 
  {\bf (Energy):}
  For $\infty\in\Sigma_{\alpha}$
  \[
  \lim_{\varepsilon\ra\infty}
  \lim_{\nu\ra\infty}
  E
  \Big(
  u^{\nu}_{\alpha};
  B_{\varepsilon}(\infty)\cap\big(\varphi_{\alpha}^{\nu}\big)^{-1}(\D)
  \Big)=0,
  \]
  where $B_{\varepsilon}(\infty)$ denotes
  set of points in $\Sigma_{\alpha}$
  with distance less than $\varepsilon$
  from $\infty$.
\end{itemize}


\section{Lagrangefy\label{sec:lagrangefy}}

The bubbling off phenomenon of holomorphic curves
is similar to the one of minimal surfaces.
A mean value inequality holds true
provided the energy is sufficiently small.
In the presence of totally real boundary conditions
the mean value inequality can be obtained
via Schwarz reflection of the energy density
with respect to a suitable local metric,
see \cite[p.\ 265]{fraun08}
and cf.\ Section \ref{subsec:meanvalineq}.
In order to obtain the removable singularity theorem
it is necessary that certain boundary terms during
partial integrations via Stokes's theorem vanish.
One way to achieve this
is a choice of a local symplectic form
that turns the boundary condition into
a Lagrangian submanifold as observed in \cite{zehm03},
see Section \ref{subsec:locsympform}.
After all the convergence modulo bubbling theorem
from \cite{mcsa04}
passes to the present situation,
see Section \ref{subsec:convmodbubb}.

Denote by $K_M\subset M$ and $K_L\subset L$
relative compact open subsets
such that $u^{\nu}(\D)\subset K_M$
and $u^{\nu}(\partial\D)\subset K_L$
for all $\nu\in\N$.

\subsection{Local symplectic form\label{subsec:locsympform}}

By \cite[p.\ 544/545]{geizeh13}
there exists a relative compact
open domain $\Omega_L$ in $M$
and a symplectic form $\omega_L$
on a neighbourhood of the closure of $\Omega_L$
such that $\Omega_L$ contains $K_L$ and
$\omega_L$
\begin{itemize}
\item
  tames the restriction of $J$ to $\Omega_L$,
\item 
  has a primitive $\lambda$ that vanishes on $L$,
  and
\item 
  turns $JT(L\cap\Omega_L)$ into a Lagrangian subbundle
  of $T\Omega_L$.
\end{itemize}
Notice,
that the metrics obtained by symmetrising
$\omega(\,.\,,J\,.\,)$ and $\omega_L(\,.\,,J\,.\,)$
restricted to $\Omega_L$
are equivalent.

The local symplectic form is
constructed in the following way:
Let $g$ be a metric on $M$
such that
$J$ is a orthogonal transformation and
the subbundle $TL$ of $TM|_L$
is orthogonal to $JTL$ along a
neighbourhood of $K_L$.
Define a non-degenerate $2$-form $\sigma:=g(J\,.\,,.\,)$
and notice that $\sigma$ tames $J$.
The normal exponential map of $g$ allows
an identification of a neighbourhood
$\Omega_L$ of $K_L$
with a neighbourhood of $K_L$ in $T^*L$,
cf.\ the proof of \cite[Theorem 3.33]{mcsa98}.
Under this identification the restriction
of $\sigma$ to $L\cap\Omega_L$ equals $\rmd\lambda$,
where $\lambda$ denotes the Liouville form
given by $T^*L$.
The claim follows by setting $\omega_L=\rmd\lambda$
and shrinking $\Omega_L$ if necessary.

\subsection{Mean value inequality\label{subsec:meanvalineq}}

We denote by $B_r(z)$
the set of complex numbers
in the closed upper half-plane $\Hp$
whose Euclidean distance to $z\in\Hp$ is
smaller than $r$.
Then there exist positive constants
$\hbar$ and $C$
such that for all holomorphic maps
\[
w\co(\Hp,\R)\lra (K_M,K_L)
\]
with
\[
\frac12\int_{B_r(z)}|\nabla w|^2
<\hbar
\]
the mean value inequality
\[
|\nabla w(z)|^2
\leq
\frac{C}{r^2}\int_{B_r(z)}|\nabla w|^2
\]
is satisfied.
This follows from
\cite[Appendix A]{fraun08} or
\cite[Section 4.3]{mcsa04} in the totally real case also
even if the results are only stated for Lagrangian boundaries.

In particular,
all non-constant holomorphic
spheres $\C P^1\ra K_M$
and all non-constant holomorphic
discs $(\D,\partial\D)\ra (K_M,K_L)$
have energy greater or equal to $\hbar$.
Indeed,
because of the conformal invariance of the energy
we can assume that $w$ is defined
on $\Hp$ or $\C$, resp.
Consider $B_r(z)$ for each point $z$ in $B_r(0)$.
If the energy of $w$ on $B_{2r}(0)$
is smaller than $\hbar$
then $\sup_{B_r(0)}|\nabla w|^2$
is bounded by the energy of $w$
on $B_{2r}(0)$ times $2C/r^2$.
The claim follows
if this inference is conclusive
for arbitrary large radii $r$.

\subsection{Convention\label{subsec:convention}}

The mean value inequality estimates
the gradient of a holomorphic map
in terms of its $L^2$-norm
provided the energy is less than $\hbar$.
Because our applications
will be geometric in nature
the holomorphic maps
to which we will apply
the inequality
will have even smaller energy.
Therefore,
we shrink $\hbar$ such that:
\begin{itemize}
\item 
  The quantity $\sqrt{8C\hbar}\;\pi$
  is smaller than the distance of $K_L$ and
  the boundary of $\Omega_L$.
\item 
  There exists a positive constant $c$
  such that for any smooth curve $\gamma$
  with endpoints $\gamma(0)$ and $\gamma(1)$
  on $K_L$ and $\length(\gamma)$
  less than or equal to $\sqrt{8C\hbar}\;\pi$
  the {\bf isoperimetric inequality}
  holds true in the sense that
  \[
  \left|
    \int_0^1\gamma^*\lambda
  \right|
  \leq
  c\length(\gamma)^2,
  \]
  as well as the corresponding
  isoperimetric inequality for
  smooth loops,
  see \cite[Theorem 4.4.1]{mcsa04}.
  Notice, that the inequality
  is invariant under reparametrisations.
\end{itemize}

\subsection{Removal of singularity\label{subsec:remofsing}}

Let $w$ be a punctured
holomorphic half-plane
\[
w\co
\big(
\Hp\setminus 0,\R\setminus 0
\big)
\lra (K_M,K_L)
\]
or a punctured holomorphic plane
$w\co\C\setminus 0\ra K_M$ of finite energy.
Then $w$ extends to a holomorphic map
on $\Hp$, resp., $\C$.

To see this notice
that
\[
E\big(w;B_r(0)\big)=\int_{B_r(0)}w^*\omega
\]
tends to zero for $r\searrow 0$.
Using polar coordinates we consider
arcs $\gamma_r(\theta)=w(r\rme^{\rmi\theta})$,
for $\theta\in[0,\pi]$,
that lie on the holomorphic half-plane $w$.
Because of
\[
|\dot\gamma_r(\theta)|^2
=\frac{r^2}{2}
\big|\nabla w(r\rme^{\rmi\theta})\big|^2
\]
and $B_r(r\rme^{\rmi\theta})\subset B_{2r}(0)$
the mean value inequality implies
\[
|\dot\gamma_r(\theta)|^2
\leq\frac{C}{2}
\int_{B_{2r}(0)}
\big|\nabla w\big|^2
\]
provided that $r$ is sufficiently small.
Therefore,
the length of $\gamma_r$ tends to zero
so that $w$ maps a neighbourhood of
$0\in\Hp$ into $\Omega_L$.

Observe that up to this point the exact
local symplectic form $\omega_L$ is not used. 
The removability
of the singularity follows
with \cite[Theorem 4.1.2]{mcsa04},
which requires a Lagrangian boundary condition
in the case of the half-plane.
Using the symplectic form $\omega_L$
on $\Omega_L$ this requirement will be satisfied.

\subsection{Convergence modulo bubbling\label{subsec:convmodbubb}}

Let $\Omega^{\nu}$ be exhausting sequence
of open subsets of $\Hp$.
Let
$
w^{\nu}\co
\big(
\Omega^{\nu},\Omega^{\nu}\cap\R
\big)
\ra (K_M,K_L)
$
be a sequence of holomorphic maps
with bounded energy.
With the results of the proceeding section
one shows as in \cite[Section 4.6]{mcsa04}:
There exists a holomorphic map
$
w\co(\Hp,\R)\ra (K_M,K_L)
$
and a finite set
$Z=\{z_1,\ldots,z_{\ell}\}$ of points
in $\Hp$,
the so-called {\bf bubble points},
such that
\begin{itemize}
\item 
  $w^{\nu}$ converges to $w$
  in $C^{\infty}_{\loc}$ on $\Hp\setminus Z$.
\item 
  For all bubble points $z_j$
  and for all $\varepsilon>0$
  such that $B_{\varepsilon}(z_j)\cap Z$
  consists of precisely the bubble point $z_j$
  the limit
  \[
  m_{\varepsilon}(z_j):=
  \lim_{\nu\ra\infty}
  E\big(w^{\nu};B_{\varepsilon}(z_j)\big)
  \]
  exists.
  Moreover the map
  $\varepsilon\mapsto m_{\varepsilon}(z_j)$
  is continuous and the {\bf mass} $m_j$
  of the bubble point $z_j$ satisfies
  \[
  m_j:=
  \lim_{\varepsilon\ra 0}
  m_{\varepsilon}(z_j)
  \geq\hbar.
  \]
\item 
  For all relative compact subsets $K$ of $\Hp$
  that contain the bubble set $Z$
  \[
  \lim_{\nu\ra\infty}E(w^{\nu};K)
  =E(w;K)+m_1+\ldots+m_{\ell}.
  \]
\end{itemize}
The analog statement with $\Hp$ replaced by $\C$
holds true.


\section{Bubbles connect\label{sec:bubbcon}}

In this section we will prove
Theorem \ref{mainthm}.
The argumentation is the same as in
\cite[Section 3.2.2]{fraun08}
and the reader is referred to the cited source.
Only modifications of the bubbling off analysis
caused by the totally real boundary condition are required.
The implementation of the modifications
is the aim of this section.

A sequence of holomorphic discs as described
in the theorem converges modulo bubbling
to a holomorphic disc,
see Section \ref{subsec:convmodbubb}.
A suitable rescaling in a neighbourhood
of any of the bubble points
yields sequences of holomorphic maps
that again converge modulo bubbling.
The proof proceeds by induction,
which terminates after finitely many steps
due to the energy bound.
Care has to be taken in view of the stability condition,
see \cite[Section 3.2.2]{fraun08}.

The rescaling near the bubble points is done
in a particular way,
see Section \ref{subsec:softrescal}.
In order to obtain a stable map for the limiting object
it is necessary to detect all the bubbles.
The main tool for that is the
concentration inequality of Section \ref{subsec:longsmall}.
It implies that the energy of the cut off annuli
that have large ratio
is concentrated near the boundary components.
In order to show that the bubbles connect
the distance-energy inequality is provided in
Section \ref{subsec:longsmall} also.
That the relative homotopy class is preserved
under Gromov convergence will follow
with the distance-energy inequality too,
see \cite[Section 3.3]{fraun08}.

\subsection{Long cylinders with small energy\label{subsec:longsmall}}

We consider a sequence of holomorphic maps
$w^{\nu}\co
\big(
B_1(0),(-1,1)
\big)\ra
(M,L)$
that are defined
on the open unit half-disc
in $\Hp$.
We assume that
$w^{\nu}\big(B_1(0)\big)\subset K_M$
and
$w^{\nu}\big((-1,1)\big)\subset K_L$
for all $\nu\in\N$.
Moreover,
we consider sequences
$\varepsilon^{\nu}$, $\delta^{\nu}$
of positive real numbes and
a sequence of complex numbers
$z^{\nu}=x^{\nu}+\rmi y^{\nu}$
on the upper half-plane
such that
\begin{itemize}
\item
  $\delta^{\nu}\leq\varepsilon^{\nu}\searrow 0$,
\item
  $\varepsilon^{\nu}/\delta^{\nu}\ra\infty$,
\item
  $z^{\nu}$ converges to $0$ in $\Hp$, and
\item
  there exists a positive constant $\eta$
  such that
  $0\leq y^{\nu}/\delta^{\nu}\leq\eta$.
\end{itemize}
We assume that
the energy
\[
e(w^{\nu}):=
E\big(w^{\nu};
A_{z^{\nu}}(\delta^{\nu},\varepsilon^{\nu})
\big)
<\hbar
\]
for all $\nu\in\N$,
where
\[
A_z(\delta,\varepsilon):=
B_{\varepsilon}(z)\setminus\overline{B_{\delta}(z)}
\subset\Hp
\]
is an open annulus
cut off along the real line.

We remark that for all $T>\ln\eta$
we have $y^{\nu}<\rme^T\delta^{\nu}$
so that both circular boundary components of
$A_{z^{\nu}}\big(\rme^T\delta^{\nu},\rme^{-T}\varepsilon^{\nu}\big)$
intersect the real axis.
Moreover,
the annulus
$A_{z^{\nu}}\big(\rme^T\delta^{\nu},\rme^{-T}\varepsilon^{\nu}\big)$
is non-empty precisely if $T<\ln\sqrt{\varepsilon^{\nu}/\delta^{\nu}}$.

\begin{lem}
  \label{lem:concentration}
  There exist positive constants
  $c$ (only depending on the geometry of $K_L$ and $K_M$),
  $T_0$ (only depending on $\eta$),
  and a natural number $\nu_0$
  such that for all $\nu\geq\nu_0$
  (we have $\ln\sqrt{\varepsilon^{\nu}/\delta^{\nu}}>T_0$)
  and for all
  $T\in\big[T_0,\ln\sqrt{\varepsilon^{\nu}/\delta^{\nu}}\big]$
  the following inequalities hold true,
  where we abbreviate $w=w^{\nu}$, $z=z^{\nu}$
  $\varepsilon=\varepsilon^{\nu}$, and
  $\delta=\delta^{\nu}$:
  \begin{itemize}
  \item 
    {\bf (Concentration inequality):}
    \[
    E\Big(w;
    A_{z}
    \big(\rme^T\delta,
    \rme^{-T}\varepsilon\big)
    \Big)
    \leq
    c\,\rme^{-T/c}e(w).
    \]
  \item 
    {\bf (Distance-energy inequality):}
    For all points $z_1,z_2$
    on the cut off annulus
    $A_{z}
    \big(\rme^T\delta,
    \rme^{-T}\varepsilon\big)$
    we have that
    \[
    \dist\big(w(z_1),w(z_2)\big)
    \leq
    c\,\rme^{-T/c}\sqrt{e(w)}.
    \]
  \end{itemize}
\end{lem}

\begin{rem}
 As the proof will show we can take
 $T_0=\max\big\{7,3+2\ln\eta\big\}$.
\end{rem}

\begin{proof}
  In order to symmetrise the problem
  we rescale
  \[
  v(\zeta):=
  w
  \big(
  \sqrt{\delta\varepsilon}\cdot\zeta
  \big)
  \]
  so that the resulting holomorphic map
  $v$ is defined on
  $B_{1/\sqrt{\delta\varepsilon}}(0)$.
  Notice, that the inverse rescaling map
  sends $A_{z}(\delta,\varepsilon)$ to
  \[
  A\big(1/R,R\big):=
  A_{z/\sqrt{\delta\varepsilon}}\big(1/R,R\big),
  \]
  where $R:=\sqrt{\varepsilon/\delta}$.
  By conformal invariance
  and invariance of the length
  under reparametrisations
  it suffices to prove
  the symmetrised variant of the lemma.
  In order to simplify the notation
  we will write $z$ instead of
  $z/\sqrt{\delta\varepsilon}$.
    
  The choice $T_0>\ln\eta$
  is understood
  so that
  for all $T\geq T_0$
  both circular boundary components
  of the cut off annuli
  $A\big(\rme^T/R,R/\rme^T\big)$
  intersect the real axis.
  We claim that for $T$ sufficiently large
  $v$ maps the annuli
  $A\big(\rme^T/R,R/\rme^T\big)$
  into $\Omega_L$.
  For that we consider the curves
  \[
  \gamma_r(\theta):=
  v\big(z+r\rme^{\rmi\theta}\big).
  \]
  In order to estimate its length
  we cover the arcs
  $\theta\mapsto z+r\rme^{\rmi\theta}$
  by the $r/2$-distance discs
  $B_{r/2}\big(z+r\rme^{\rmi\theta}\big)$
  of $\Hp$.
  We require that
  $r\in\big(2/R,2R/3\big)$
  so that the cut off discs
  $B_{r/2}\big(z+r\rme^{\rmi\theta}\big)$
  are contained in
  $A\big(1/R,R\big)$.
  Hence, the energy of $v$
  restricted to the
  $r/2$-distance discs
  is bounded by $\hbar$.
  The mean value inequality implies
  \[
  |\dot\gamma_r(\theta)|^2
  \leq 2C
  \int_{A(1/R,R)}
  \big|\nabla v\big|^2.
  \]
  Therefore,
  the length of $\gamma_r$
  is bounded by $2\pi\sqrt{2C\hbar}$.
  We choose $T_0>\ln 2$ so that
  $A\big(\rme^T/R,R/\rme^T\big)$
  is contained in
  $A\big(2/R,2R/3\big)$
  for all $T\geq T_0$.
  This implies that
  \[
  \length(\gamma_r)
  \leq
  \sqrt{8C\hbar}\;\pi
  \]
  for all
  $r\in\big(\rme^T/R,R/\rme^T\big)$.
  The curves $\gamma_r$ connect
  interior points of
  the holomorphic square
  \[
  v\Big(A\big(\rme^T/R,R/\rme^T\big)\Big)
  \]
  to the boundary in $K_L$
  provided $T\geq T_0$.
  By the assumptions on the quantity $\hbar$,
  see Section \ref{subsec:convention},
  the holomorphic square
  is contained in
  $\Omega_L$ as claimed.
  
  The closure of $\Omega_L$ is compact.
  Hence, up to a change of constants
  in the formulation of the lemma
  it suffices to prove the inequalities
  with respect to the symplectic form
  $\omega_L=\rmd\lambda$
  and the metric given by the
  symmetrisation of $\omega_L(\,.\,,J\,.\,)$.
  
  In view of the first inequality
  we will use logarithmic polar coordinates.
  We reparametrise
  \[
  u(s+\rmi t)=
  v\big(z+\rme^{s+\rmi t}\big).
  \]
  The preimage $D_T$ of
  $A\big(\rme^T/R,R/\rme^T\big)$
  in the square
  \[
  \big(T-\ln R,\ln R-T\big)
  \times
  \left
    (-\frac{\pi}{2},\frac32\pi
  \right)
  \]
  is bounded by the
  graphs of the functions $s\mapsto t_0(s),t_1(s)$,
  which are mapped into $K_L$ by $u$.
  Therefore, we can express the energy
  \[
  e(T)=\int_{D_T}u^*\rmd\lambda
  \]
  in terms of $\gamma_s(t)=u(s+\rmi t)$ as
  \[
  e(T)=
  \int\big(\gamma_{\ln R-T}\big)^*\lambda-
  \int\big(\gamma_{T-\ln R}\big)^*\lambda,
  \]
  where the first integral is taken
  over all $t$ in
  $\big(t_0(\ln R-T),t_1(\ln R-T)\big)$;
  the second over all $t$ in
  $\big(t_0(T-\ln R),t_1(T-\ln R)\big)$.
  With the isoperimetric inequality,
  see Section \ref{subsec:convention},
  this yields
  \[
  e(T)\leq
  c_1
  \left(
    \length\big(\gamma_{\ln R-T}\big)^2+
    \length\big(\gamma_{T-\ln R}\big)^2
  \right)
  \]
  for a universal constant $c_1>0$.
  Using the Cauchy-Schwarz inequality
  and the Cauchy-Riemann equation
  we obtain
  \[
  e(T)\leq
  2\pi c_1
  \left(
    \int
    \Big|u_s\big(\ln R-T+\rmi t\big)
    \Big|^2
    \rmd t+
    \int
    \Big|u_s\big(T-\ln R+\rmi t\big)
    \Big|^2
    \rmd t
  \right),
  \]
  where the first integral is taken over
  $\big(t_0(\ln R-T),t_1(\ln R-T)\big)$;
  the second over
  $\big(t_0(T-\ln R),t_1(T-\ln R)\big)$.
  Taking the derivative of
  \[
  e(T)=
  \int_{T-\ln R}^{\ln R-T}
  \left(
    \int_{t_0(s)}^{t_1(s)}
    \big|u_s(s+\rmi t)\big|^2
    \rmd t
  \right)
  \rmd s
  \]
  with respect to $T$ yields
  \[
  \frac{\rmd}{\rmd T}e(T)
  \leq
  -\frac{1}{2\pi c_1}e(T)
  \]
  for all $T\geq T_0$.
  Gr\"onwall's inequality implies
  \[
  e(T)\leq
  \rme^{-\frac{1}{2\pi c_1}(T-T_0)}e(T_0).
  \]
  Because $e(T_0)$ is bounded by $e(w)$
  up to a universal constant
  this proves the concentration inequality;
  the first inequality of the lemma.
  
  We prove the second.
  Observe that each two points
  \[
  z_1=z+r_1\rme^{\rmi\theta_1},
  \qquad
  z_2=z+r_2\rme^{\rmi\theta_2}
  \]
  in $A\big(\rme^T/R,R/\rme^T\big)$
  can be joint by a circular arc
  provided $r_1=r_2$, or,
  if $r_1\neq r_2$,
  by a path consisting of a circular arc
  connecting $z_1$ with $z+r_1\rmi$,
  a line segment connecting $z+r_1\rmi$
  with $z+r_2\rmi$, and a circular arc connecting
  $z+r_2\rmi$ with $z_2$.
  Therefore, it is possible to prove
  the second inequality
  with help of circular arcs
  and line segments in $z+\R\rmi$ only.
  
  We estimate the length of the path
  \[
  I\ni\theta\longmapsto
  v\big(z+r\rme^{\rmi\theta}\big),
  \]
  where $I$ is a subinterval of
  $\big(\!-\frac{\pi}{2},\frac32\pi\big)$
  and $r\in\big(\rme^T/R,R/\rme^T\big)$.
  By the mean value inequality
  \[
  \left|
    \frac{\partial v}{\partial\theta}
    \big(z+r\rme^{\rmi\theta}\big)
  \right|
  =
  \frac{r}{\sqrt2}
  \Big|
    \nabla v
    \big(z+r\rme^{\rmi\theta}\big)
  \Big|
  \]
  is bounded by
  \[
  \sqrt{2C}
  \left(
    \int_{B_{r/2}\big(z+r\rme^{\rmi\theta}\big)}
    \big|\nabla v\big|^2
  \right)^{1/2}.
  \]
  For all $T\geq\ln2+T_0$ we get with
  $S:=T-\ln2$ that
  $B_{r/2}\big(z+r\rme^{\rmi\theta}\big)$
  is contained in
  \[
  A\left(
    \frac12\frac{\rme^T}{R},\frac32\frac{R}{\rme^T}
  \right)=
  A\left(
    \frac{\rme^S}{R},\frac34\frac{R}{\rme^S}
  \right)  
  \]
  the latter being contained in
  $A\big(\rme^{T_0}/R,R/\rme^{T_0}\big)$.
  Hence,
  we can estimate the upper bound by
  \[
  2\sqrt{C}
  \sqrt{
    E\Big(v;A\big(\rme^S/R,R/\rme^S\big)\Big)
  },
  \]
  which by the concentration inequality
  is bounded by
  \[
  2\sqrt{C}
  \sqrt{c\,\rme^{-S/c}e(w)}.
  \]
  Consequently,
  replacing $T_0$ by $T_0+\ln2$
  we find a constant $c_2>0$
  such that
  \[
  \left|
    \frac{\partial v}{\partial\theta}
    \big(z+r\rme^{\rmi\theta}\big)
  \right|
  \leq
  c_2\rme^{-T/c_2}
  \sqrt{e(w)}
  \]
  for all $T\geq T_0$.
  Integrating $\theta$ over the subinterval $I$
  proves the distance-energy inequality
  for circular paths.
  
  We prove the distance-energy inequality
  along line segments
  $r\mapsto z+r\rmi$
  with
  $r\in\big(\rme^T/R,R/\rme^T\big)$.
  Using logarithmic polar coordinates
  we write
  \[
  u(s+\rmi t)=v\big(z+\rme^{s+\rmi t}\big),
  \]
  where $|s|<\ln R-T$ and $t\in(0,\pi)$.
  We choose $T_0>\pi/2$.
  The mean value inequality implies that
  \[
  \left|
    \frac{\partial u}{\partial s}
    \big(s+\rmi\pi/2\big)
  \right|
  =
  \frac{1}{\sqrt2}
  \Big|
    \nabla u
    \big(s+\rmi\pi/2\big)
  \Big|
  \]
  is bounded by
  \[
  \frac{\sqrt{2C}}{\pi}
  \left(
    \int_{B_{\pi/2}\big(s+\rmi\pi/2\big)}
    \big|\nabla u\big|^2
  \right)^{1/2}.
  \]
  Notice, that
  $B_{\pi/2}\big(s+\rmi\pi/2\big)$
  is contained in the square
  \[
  \Big(\!-|s|-\pi/2,|s|+\pi/2\Big)\times(0,\pi).
  \]
  The image under the inverse
  conformal coordinate transformation
  is contained in
  \[
  A:=A\big(\rme^{-|s|-\pi/2},\rme^{|s|+\pi/2}\big)
  =A\Big(\rme^{\ln R-|s|-\pi/2}/R,R/\rme^{\ln R-|s|-\pi/2}\Big).
  \]
  Therefore,
  \[
  \left|
    \frac{\partial u}{\partial s}
    \big(s+\rmi\pi/2\big)
  \right|
  \leq
  \frac{2}{\pi}\sqrt{C}\sqrt{E(v;A)}.
  \]
  Further, we are allowed to apply
  the concentration inequality
  for all $T\geq T_0+\pi/2$.
  Indeed, we get $|s|<\ln R-T_0-\pi/2$
  such that
  \[
  A\subset
  A\big(\rme^{T_0}/R,R/\rme^{T_0}\big).
  \]
  Hence, we obtain
  \[
  \left|
    \frac{\partial u}{\partial s}
    \big(s+\rmi\pi/2\big)
  \right|
  \leq
  \frac{2}{\pi}\sqrt{C}
  \sqrt{c\,\rme^{\big(|s|+\pi/2-\ln R\big)/c}e(w)}.
  \]
  An integration of the inequality over all
  $s\in\big(T-\ln R,\ln R-T\big)$
  using the symmetry of the integrand
  on the right hand side shows
  that the length of the path
  $s\mapsto u\big(s+\rmi\pi/2\big)$
  is bounded by
  \[
  \frac{2}{\pi}\sqrt{C}
  \sqrt{ce(w)}
  4c\,\rme^{\big(\pi/2-T\big)/2c}.
  \]
  Replacing $T_0$ by $T_0+\pi/2$
  completes the proof of the lemma.
\end{proof}

\subsection{Soft rescaling\label{subsec:softrescal}}

Consider the open disc $B_r(z_0)$
of radius $r>0$ in $\Hp$
with center equal to $z_0\in\Hp$
and a sequence
$
w^{\nu}\co
\big(B_r(z_0),B_r(z_0)\cap\R\big)
\ra (K_M,K_L)
$
of holomorphic maps with bounded energy.
Assume that there exists a pointed limiting
holomorphic map
$
w\co
\big(B_r(z_0),B_r(z_0)\cap\R\big)
\ra (K_M,K_L)
$
in the sense that
\begin{itemize}
\item
  $w^{\nu}$ converges to $w$
  in $C^{\infty}_{\loc}$
  on $B_r(z_0)\setminus\{z_0\}$ and
\item
  the limit
  \[
  m_0=
  \lim_{\varepsilon\ra 0}
  \lim_{\nu\ra\infty}
  E\big(w^{\nu};B_{\varepsilon}(z_0)\big),
  \]
  the so-called {\bf mass} at $z_0$,
  exists and is greater or equal to $\hbar$.
\end{itemize}
The aim of this section is to describe
the behavior of the sequence $w^{\nu}$
near the selected bubble point $z_0$.
We will use the so-called soft rescaling technique
that was introduced in \cite{hosa95}.
Denote by $\Sigma$
the closed unit disc $\D=\Hp\cup\infty$
or $\C P^1=\C\cup\infty$
depending on whether $z_0$ is real
or has positive imaginary part.

\begin{lem}
  There exists
  a sequence $\varphi^{\nu}$
  of M\"obius transformations of $\Sigma$,
  a holomorphic map
  $
  v\co
  \big(\D,\partial\D\big)
  \ra (K_M,K_L)
  $, resp.,
  $
  v\co\C P^1\ra K_M
  $, and
  a finite subset $Z=\{z_1,\ldots,z_{\ell}\}$
  of bubble points on $\Sigma\setminus\infty$
  such that
  \begin{itemize}
  \item 
    A subsequence of
    $v^{\nu}:=w^{\nu}\circ\varphi^{\nu}$
    converges to $v$ in $C^{\infty}_{\loc}$
    on $\Sigma\setminus\big(Z\cup\infty\big)$
    and
    $\varphi^{\nu}$ converges to $0$
    in $C^{\infty}_{\loc}$ on $\Sigma\setminus\infty$.
  \item 
    If $v$ is constant, then the following
    {\bf stability condition} is satisfied:
    Either $\#Z\geq2$,
    regardless whether $\Sigma$ is a disc or a sphere,
    or $\#Z<2$, in which case
    $\Sigma$ is the closed unit disc
    and there is precisely one bubble point
    in $Z$,
    which lies in the interior.
  \item 
    For all bubble points $z_j$
    the limit
    \[
    m_j:=
    \lim_{\varepsilon\ra 0}
    \lim_{\nu\ra\infty}
    E\big(v^{\nu};B_{\varepsilon}(z_j)\big),
    \]
    the so-called {\bf mass} at $z_j$,
    exists and is greater or equal to $\hbar$.
  \item 
    Denote by
    \[
    m_{\infty}:=
    \lim_{R\ra\infty}
    \lim_{\nu\ra\infty}
    E\big(v^{\nu};\Sigma\setminus B_R(0)\big)
    \]
    the {\bf mass at infinity}.
    Then
    \[
    \lim_{\nu\ra\infty}E(w^{\nu})
    =m_0+m_{\infty}
    \]
    and
    \[
    m_0=E(v)+m_1+\ldots+m_{\ell}
    \]
    so that {\bf no energy is lost}.
  \item 
    The {\bf bubbles connect}, i.e., $w(0)=v(\infty)$.
  \end{itemize}
\end{lem}

\begin{proof}
  In the case the distinguished bubble point
  $z_0$ does not lie on the real axis the proof
  can be found in \cite[Section 4.7]{mcsa04}.
  In the alternative case we will translate
  the arguments from \cite[Section 3.2.1]{fraun08}
  to the present situation.
  The modifications are caused by
  the totally real boundary condition.
  
  We can assume that $z_0=0$ and $r=2$.
  The sequence
  \[
  z^{\nu}=x^{\nu}+\rmi y^{\nu}
  \]
  of points on which $|\nabla w^{\nu}|$
  attains its maximum on $B_1(0)$
  converges to the origin.
  For $\nu$ sufficiently large we find $\delta^{\nu}>0$
  such that
  \[
  E\big(w^{\nu};B_{\delta^{\nu}}(z^{\nu})\big)
  =m_0-\hbar/2.
  \]
  By the definition of $m_0$
  the sequence $\delta^{\nu}$ must converge to $0$.
  We will consider two cases.
  Either a subsequence of $y^{\nu}/\delta^{\nu}$
  converges to a real number or tends to infinity.
  
  We start with the first case.
  We can assume that $y^{\nu}/\delta^{\nu}$
  converges to $\eta\geq0$.
  Setting
  \[
  \varphi^{\nu}(z):=
  x^{\nu}+\delta^{\nu}z
  \]
  for all $\nu$ we obtain a sequence of
  M\"obius transformations,
  which preserve the upper half-plane.
  The sequence $\varphi^{\nu}$ tends to $0$
  in $C^{\infty}_{\loc}$.
  With Section \ref{subsec:convmodbubb}
  the sequence
  \[
  v^{\nu}:=w^{\nu}\circ\varphi^{\nu}
  \]
  converges modulo bubbling in the complement
  $\Hp\setminus Z$ of finitely many bubble points
  $Z=\{z_1,\ldots,z_{\ell}\}$ to $v$.
  The masses exists at all bubble points $z_j$.
  By Section \ref{subsec:remofsing}
  $v$ extends to a holomorphic disc.
  
  By construction
  \[
  E\Big(v^{\nu};B_1\big(\rmi y^{\nu}/\delta^{\nu}\big)\Big)
  =m_0-\hbar/2.
  \]
  The definition of $m_0$ implies
  that for $\varepsilon>0$ sufficiently small
  and for $\nu\in\N$ sufficiently large
  \[
  E\Big(v^{\nu};B_{\varepsilon/\delta^{\nu}}\big(-x^{\nu}/\delta^{\nu}\big)\Big)
  \leq m_0+\hbar/4.
  \]
  Because the masses are greater or equal than $\hbar$
  the bubble points $z_1,\ldots,z_{\ell}$ are contained
  in the closure of $B_1(\rmi\eta)$.
  Moreover, because the supremum of $|\nabla v^{\nu}|$ over
  $B_{1/\delta^{\nu}}\big(-x^{\nu}/\delta^{\nu}\big)$
  is attained at $\rmi y^{\nu}/\delta^{\nu}$
  the point $\rmi\eta$ is a bubble point
  provided $Z$ is not empty.
  
  We claim that
  \[
  \lim_{R\ra\infty}
  \lim_{\nu\ra\infty}
  E\Big(v^{\nu};B_R\big(\rmi y^{\nu}/\delta^{\nu}\big)\Big)
  =m_0.
  \]
  Before we prove this we will draw its consequences.
  First of all it implies that for $s>1$ the sum
  \[
  E\big(v;\Hp\setminus B_s(\rmi\eta)\big)+
  \lim_{\nu\ra\infty}
  E\Big(v^{\nu};B_s\big(\rmi y^{\nu}/\delta^{\nu}\big)\Big)
  \]
  is equal to $m_0$.
  Cutting out small neighbourhoods of the bubble points
  $z_1,\ldots,z_{\ell}$ the limit equals the sum of the masses.
  This shows that
  \[
  E(v)+m_1+\ldots+m_{\ell}=m_0.
  \]
  In particular, $\rmi\eta$ is a bubble point
  if $v$ is constant.
  A similar argument and
  conformal invariance of the energy imply
  \[
  \lim_{\nu\ra\infty}E(w^{\nu})=
  E(v)+m_1+\ldots+m_{\ell}+m_{\infty}.
  \]
  In order to verify the stability condition
  observe that
  \[
  \lim_{\nu\ra\infty}
  E\big(v^{\nu};B_R(\rmi\eta)\big)
  \]
  is independent of $R>1$
  provided that $v$ is constant.
  Hence,
  \[
  \lim_{\nu\ra\infty}
  E\Big(v^{\nu};B_R\big(\rmi y^{\nu}/\delta^{\nu}\big)\Big)
  =m_0
  \]
  for all $R>1$.
  Because
  \[
  E\Big(v^{\nu};B_1\big(\rmi y^{\nu}/\delta^{\nu}\big)\Big)
  =m_0-\hbar/2
  \]
  for all $\nu$ and in view of the above cutting argument
  for the bubble points,
  which are contained in the closure of $B_1(\rmi\eta)$,
  there must be a bubble point on the round boundary arc
  of $B_1(\rmi\eta)$.
  In other words, $\#Z\geq2$ provided $v$ is constant.
  
  We prove the above claim,
  namely that
  \[
  \lim_{R\ra\infty}
  \lim_{\nu\ra\infty}
  E\big(w^{\nu};B_{R\delta^{\nu}}(z^{\nu})\big)
  =m_0.
  \]
  Observe that for all $R\geq1$
  \[
  m_0-\hbar/2\leq
  \lim_{\nu\ra\infty}
  E\big(w^{\nu};B_{R\delta^{\nu}}(z^{\nu})\big)
  \leq m_0.
  \]
  Arguing by contradiction we suppose that
  there exists $\mu>0$ such that for all $R\geq1$
  \[
  \lim_{\nu\ra\infty}
  E\big(w^{\nu};B_{R\delta^{\nu}}(z^{\nu})\big)
  \leq m_0-\mu.
  \]
  In order to lead this to a contradiction
  we choose a sequence $\varepsilon^{\nu}\searrow 0$
  such that
  \[
  \lim_{\nu\ra\infty}
  E\big(w^{\nu};B_{\varepsilon^{\nu}}(z^{\nu})\big)
  =m_0,
  \]
  see \cite[p.\ 234]{fraun08}.
  The rescaling argument from \cite[p.\ 234]{fraun08}
  shows that
  \[
  \lim_{\nu\ra\infty}
  E\Big(w^{\nu};A_{z^{\nu}}\big(\rme^{-T}\varepsilon^{\nu},\varepsilon^{\nu}\big)\Big)
  =0
  \]
  for all $T>0$.
  Hence,
  \[
  m_0=
  \lim_{\nu\ra\infty}
  E\big(w^{\nu};B_{\rme^T\delta^{\nu}}(z^{\nu})\big)+
  \lim_{\nu\ra\infty}
  E\Big(w^{\nu};A_{z^{\nu}}\big(\rme^T\delta^{\nu},\rme^{-T}\varepsilon^{\nu}\big)\Big).
  \]
  Moreover, by the definitions of $\delta^{\nu}$ and $\varepsilon^{\nu}$
  we have
  \[
  \lim_{\nu\ra\infty}
  E\big(w^{\nu};A_{z^{\nu}}(\delta^{\nu},\varepsilon^{\nu})\big)
  =\hbar/2,
  \]
  as well as,
  using the contradiction assumption,
  that the ratio $\varepsilon^{\nu}/\delta^{\nu}$
  tends to infinity.
  The concentration inequality of Lemma \ref{lem:concentration}
  implies that
  \[
  \lim_{\nu\ra\infty}
  E\Big(w^{\nu};A_{z^{\nu}}\big(\rme^T\delta^{\nu},\rme^{-T}\varepsilon^{\nu}\big)\Big)
  \leq c\rme^{-T/c}\hbar/2
  \]
  for all $T$ sufficiently large.
  Combining this with the above limit-decomposition of $m_0$
  we obtain
  \[
  m_0\leq m_0-\mu+c\rme^{-T/c}\hbar/2.
  \]
  Letting $T$ tend to infinity this yields the desired contradiction.
  
  A similar argumentation to \cite[Lemma 3.6]{fraun08}
  using the distance-energy inequality of Lemma \ref{lem:concentration}
  shows that the bubbles $w$ and $v$ connect.
  This finishes the proof of the first case.
  
  For the second we assume that $y^{\nu}/\delta^{\nu}$
  tends to infinity.
  We set
  \[
  \varphi^{\nu}(z):=
  x^{\nu}+y^{\nu}z
  \]
  and obtain a sequence of
  M\"obius transformations preserving the upper half-plane
  that converges to zero in $C^{\infty}_{\loc}$.
  We consider the sequence of holomorphic maps
  obtained by
  \[
  v^{\nu}:=w^{\nu}\circ\varphi^{\nu}.
  \]
  We claim that $v^{\nu}$ converges
  modulo bubbling on $\Hp\setminus\rmi$
  to a constant map.
  Notice, that only $\rmi$ can be a bubble point.
  Indeed, we have
  \[
  E\big(v^{\nu};B_{\delta^{\nu}/y^{\nu}}(\rmi)\big)
  =m_0-\hbar/2
  \]
  for all $\nu$.
  By the definition of $m_0$
  we find $\varepsilon>0$ sufficiently small
  and $\nu$ sufficiently large such that  
  \[
  E\big(w^{\nu};B_{\varepsilon}(0)\big)
  \leq
  m_0+\hbar/4.
  \]
  The term on the left equals
  the energy of $v^{\nu}$ on the distance disc
  \[
  B_{\varepsilon/y^{\nu}}\big(-x^{\nu}/y^{\nu}\big).
  \]
  Hence, for all $\nu$ sufficiently large
  the energy of $v^{\nu}$ on
  \[
  D^{\nu}:=
  B_{\varepsilon/y^{\nu}}\big(-x^{\nu}/y^{\nu}\big)
  \setminus
  B_{\delta^{\nu}/y^{\nu}}(\rmi)
  \]
  is less than or equal to $3\hbar/4$.
  Notice,
  that the real boundary points
  $-x^{\nu}/y^{\nu}\pm\varepsilon/y^{\nu}$
  of $B_{\varepsilon/y^{\nu}}\big(-x^{\nu}/y^{\nu}\big)$
  are of different sign
  and unbounded
  as $x^{\nu}$ and $y^{\nu}$ tend to zero.
  Hence, $D^{\nu}$ approaches all of $\Hp\setminus\rmi$.
  Because the mass of each bubble point
  is greater or equal to $\hbar$
  there is no bubble point on $\Hp\setminus\rmi$.
  Moreover,
  for all $R>1$ we find $\nu$ such that
  $A_{\rmi}(1/R,R)$ is contained in the domain $D^{\nu}$.
  Hence, taking the limit for $\nu\ra\infty$ we get 
  \[
  E\big(v;A_{\rmi}(1/R,R)\big)\leq 3\hbar/4
  \]
  for all $R>1$.
  Taking $R\ra\infty$ we get that $E(v;\Hp)$
  is bounded by $3\hbar/4$.
  The mean value inequality as in
  Section \ref{subsec:meanvalineq}
  implies that $v$ must be constant.
  Finally, suppose that $v^{\nu}$ converges
  in $C^{\infty}_{\loc}$ on the whole
  closed upper half-plane.
  Then the energy $E\big(v^{\nu};B_1(\rmi)\big)$,
  which is bounded by
  \[
  E\big(v^{\nu};B_{\delta^{\nu}/y^{\nu}}(\rmi)\big)
  =m_0-\hbar/2
  \]
  from below, would trend to zero.
  But this is not possible.
  Consequently,
  $v^{\nu}$ converges to a constant map
  on $\Hp\setminus\rmi$
  and $\rmi$ is a bubble point.
  
  Denote the mass of $v^{\nu}$ at $\rmi$ by $m_1$.
  Because
  \[
  E\big(v^{\nu};B_{\varepsilon}(\rmi)\big)
  \geq
  E\big(v^{\nu};B_{\delta^{\nu}/y^{\nu}}(\rmi)\big)
  \]
  for sufficiently large $\nu$
  we see taking the limits $\nu\ra\infty$ and $\varepsilon\ra 0$
  successively that
  \[
  m_1\geq m_0-\hbar/2.
  \]
  We claim that $m_0=m_1$ so that no energy gets lost.
  First of all observe that by the definition of $m_0$
  \[
  \lim_{\varepsilon\ra 0}
  \lim_{\nu\ra\infty}
  E\big(w^{\nu};B_{\varepsilon}(z^{\nu})\big)
  =m_0.
  \]
  With the decomposition
  \[
  E\big(v^{\nu};B_{\varepsilon}(\rmi)\big)+
  E\Big(v^{\nu};A_{\rmi}(\varepsilon,\varepsilon/y^{\nu})\Big)
  \]
  of
  \[
  E\big(v^{\nu};B_{\varepsilon/y^{\nu}}(\rmi)\big)=
  E\big(w^{\nu};B_{\varepsilon}(z^{\nu})\big)
  \]
  this implies that
  \[
  m_0=m_1+
  \lim_{\varepsilon\ra 0}
  \lim_{\nu\ra\infty}
  E\Big(v^{\nu};A_{\rmi}(\varepsilon,\varepsilon/y^{\nu})\Big).
  \]
  In particular, because $m_0-m_1\leq\hbar/2$
  the double limes term
  is bounded by $\hbar/2$.
  Hence, there exists $\varepsilon_0\in(0,1/4)$
  such that for all $\varepsilon\in (0,2\varepsilon_0]$
  there exists $\nu_0\in\N$
  such that for all $\nu\geq\nu_0$
  we have $4\leq1/y^{\nu}$
  and
  \[
  E\Big(v^{\nu};A_{\rmi}(\varepsilon,\varepsilon/y^{\nu})\Big)
  \leq 2\hbar/3.
  \]
  Notice, that
  \[
  \lim_{\nu\ra\infty}
  E\big(v^{\nu};A_{\rmi}(\varepsilon,1)\big)
  =0
  \]
  for all $\varepsilon\in (0,1)$,
  because $v^{\nu}$ converges uniformly
  on $A_{\rmi}(\varepsilon,1)$
  to the constant map.
  This implies
  \[
  \lim_{\varepsilon\ra 0}
  \lim_{\nu\ra\infty}
  E\big(v^{\nu};A_{\rmi}(\varepsilon,1)\big)
  =0.
  \]
  In other words it suffices to show that
  \[
  \lim_{\varepsilon\ra 0}
  \lim_{\nu\ra\infty}
  E\Big(v^{\nu};A_{\rmi}(1,\varepsilon/y^{\nu})\Big)
  =0.
  \]
  In order to do so observe
  that for all $\varepsilon\in (0,2\varepsilon_0]$
  the annuli
  $A_{\rmi}(1,\varepsilon/y^{\nu})$
  support energy $E(v^{\nu};\,.\,)\leq 2\hbar/3$
  for all $\nu\geq\nu_0$
  and that the distance discs
  $B_{r/2}\big(\rmi+r\rme^{\rmi\theta})$
  are contained in
  $A_{\rmi}(2\varepsilon_0,2\varepsilon_0/y^{\nu})$
  for all $r\in[1,\varepsilon_0/y^{\nu}]$.
  Therefore,
  as in the first part of the proof of
  the concentration inequality of
  Lemma \ref{lem:concentration},
  again invoking the mean value inequality
  and the conventions about $\hbar$
  from Section \ref{subsec:convention},
  we get
  \[
  v^{\nu}\Big(A_{\rmi}(1,\varepsilon/y^{\nu})\Big)
  \subset\Omega_L
  \]
  for all $\varepsilon\in (0,\varepsilon_0)$,
  because both boundary components of
  $A_{\rmi}(1,\varepsilon/y^{\nu})$
  intersect the real axis.
  Hence, we are allowed to work
  with the symplectic form $\omega_L=\rmd\lambda$
  and can proceed as in the second part of the proof of
  the concentration inequality.
  Consider the rescaled maps
  \[
  u^{\nu}:=v^{\nu}\circ\sqrt{\varepsilon/y^{\nu}}
  \]
  that support energy less than $\hbar$ on
  \[
  A_{\rmi\sqrt{y^{\nu}/\varepsilon}}
  \Big(\sqrt{y^{\nu}/\varepsilon},\sqrt{\varepsilon/y^{\nu}}\,\Big).
  \]
  With $\rme^{-T}=\sqrt{\varepsilon/\varepsilon_0}$
  and the concentration inequality we get
  \[
  E\Big(u^{\nu};A\big(\rme^T\sqrt{y^{\nu}/\varepsilon_0}\,,\rme^{-T}/\sqrt{y^{\nu}/\varepsilon_0}\,\big)\Big)
  \leq c\,\rme^{-T/c}\hbar.
  \]
  Notice, that the dependence of the center
  $\rme^T\sqrt{y^{\nu}/\varepsilon_0}$ of the annulus on $T$
  does not affect the inequality
  because the proof uses logarithmic
  polar coordinates around $\rme^T\sqrt{y^{\nu}/\varepsilon_0}$.
  Undoing the rescaling we get that there exists
  a positive constant $c_2$ such that
  \[
  E\Big(v^{\nu};A_{\rmi}(1,\varepsilon/y^{\nu})\Big)
  \leq c_2\big(\sqrt{\varepsilon}\,\big)^{1/c}  
  \]
  for all $\nu$ sufficiently large.
  Consequently,
  \[
  \lim_{\nu\ra\infty}
  E\Big(v^{\nu};A_{\rmi}(1,\varepsilon/y^{\nu})\Big)
  \leq c_2\big(\sqrt{\varepsilon}\,\big)^{1/c}.
  \]
  Hence, $m_0=m_1$ follows letting $\varepsilon$
  tend to zero.
  This finishes the second part
  because the proof that the bubbles
  connect is the same as in \cite[Lemma 3.6]{fraun08}.
  This is due to the distance-energy inequality,
  see Lemma \ref{lem:concentration}.
  The proof is complete.
\end{proof}


\section{Generalizations\label{sec:general}}

In this section we collect some generalizations
of Theorem \ref{mainthm}.

\subsection{Variation of almost complex structures}

The theorem continues to hold if we allow
the almost complex structure to vary.
One can consider
a sequence of $J^{\nu}$-holomorphic discs
as well,
where $J^{\nu}$ is a
sequence of almost complex structures
that converges in $C^{\infty}_{\loc}$
to $J$.

\subsection{Uniqueness}

The limit stable map of a Gromov converging sequence
of holomorphic maps is unique up to equivalence.
A description can be found in
\cite[Theorem 3.4]{fraun08}.
The images of the Gromov converging holomorphic maps
converge in the sense of Hausdorff
to the image of the limit stable map,
see \cite[Proposition 3.2]{fraun08}
or \cite[Theorem A.1]{hosa95}.

\subsection{Marked stable maps}

For bounded sequences of stable maps
with bounded energy the notion of Gromov convergence
can be defined as in \cite[Section 4]{fraun08}.
The appearance of marked points and
the variation of almost complex structures are allowed.
The analog of Theorem \ref{mainthm}
for marked stable maps of genus zero
can be obtained as in \cite[Section 4]{fraun08}
with the modifications of the present work
caused by the totally real boundary condition.
In other words,
all results from \cite{fraun08} carry over
to the situation described in the introduction.


\begin{ack}
  We thank Peter Albers, Fr\'ed\'eric Bourgeois, Hansj\"org Geiges, and
  Klaus Niederkr\"uger
  who encouraged us to write this article.
  We thank Urs Fuchs who had explained us his approach
  to Gromov compactness of his thesis \cite{ursfu14}.
\end{ack}


\end{document}